\documentclass[preprint]{elsarticle}
\usepackage{amsmath,amssymb,amsthm}

\newcommand{\pb}{{\boldsymbol{p}}}
\newcommand{\qb}{{\boldsymbol{q}}}

\newcommand{\bv}{{\boldsymbol{v}}}

\newcommand{\C}{{\mathbb C}}
\newcommand{\D}{{\mathbb D}}
\newcommand{\N}{{\mathbb N}}

\newcommand{\R}{{\mathbb R}}
\newcommand{\T}{{\mathbb T}}
\newcommand{\Z}{{\mathbb Z}}

\theoremstyle{plain}
\newtheorem{Theorem}{Theorem}[section]

\newtheorem{Lemma}[Theorem]{Lemma}
\newtheorem{Proposition}[Theorem]{Proposition}

\theoremstyle{definition}
\newtheorem{Definition}[Theorem]{Definition}

\newtheorem{Example}[Theorem]{Example}

\journal{Journal of Approximation Theory}

\begin{document}
\begin{frontmatter}

\title{System theory and orthogonal multi-wavelets}

\author[MariaAddress]{Maria Charina
\corref{mycorrespondingauthor}}
\cortext[mycorrespondingauthor]{Corresponding author}
\ead{maria.charina@univie.ac.at}
\address[MariaAddress]{Fakult\"at f\"ur
	Mathematik, University of Vienna, 1090 Vienna, Austria}

\author[CostanzaAddress]{Costanza Conti}
	\ead{costanza.conti@unifi.it}
	\address[CostanzaAddress]{DIEF, Universit\`a di Firenze, Viale Morgagni 40/44, 50134 Firenze, Italy}

\author[MariantoAddress]{Mariantonia Cotronei}
\ead{mariantonia.cotronei@unirc.it}
\address[MariantoAddress]{DIIES, Universit\`a Mediterranea di Reggio Calabria, Via Graziella, loc. Feo di Vito, 89122 Reggio Calabria, Italy}

\author[MihaiAddress]{Mihai Putinar}
\ead{mputinar@math.ucsb.edu}
\address[MihaiAddress]{University of California at Santa Barbara, USA, and University of Newcastle,
	Newcastle upon Tyne, UK}

\begin{abstract}
In this paper we provide a complete and unifying characterization of compactly supported univariate scalar  orthogonal  wavelets and vector-valued or matrix-valued orthogonal multi-wavelets.
This characterization is based on classical results from system theory and basic linear algebra. In particular, we show that the corresponding wavelet and multi-wavelet masks are  identified with a transfer function $$
 F(z)=A+B z (I-Dz)^{-1} \, C, \quad  z \in \D=\{z \in \C \ : \ |z| < 1\},
$$
of a conservative linear system. The complex matrices $A,\ B, \ C, \ D$ define a block circulant unitary matrix. Our results show that there are no intrinsic differences  between the elegant wavelet construction by Daubechies or any other construction of vector-valued  or matrix-valued multi-wavelets. The structure of the unitary matrix defined by $A,\ B, \ C, \ D$ allows us to
parametrize in a systematic way all classes of possible wavelet and multi-wavelet masks together with the masks of the corresponding refinable functions.
\end{abstract}



\end{frontmatter}

\section{Introduction and notation}  \label{sec:intro_notation}

There has been a multitude of results on orthogonal wavelet and multi-wavelet constructions and on the characterization of the corresponding filterbank systems, since the pioneering work \cite{Daub_paper}. Our goal is to unify all those approaches. Indeed we show that there are no intrinsic differences between the elegant construction of wavelets by Daubechies, in the scalar case, or any other construction of vector or matrix-valued
multi-wavelets.
In particular, we
compare our results with the recent characterization of orthogonal multi-wavelets in \cite{AlpayJorgensen}.

By \cite{RS95}, the constructions of compactly supported multi-wavelets via the Unitary Extension Principle  boil down to manipulations with certain trigonometric polynomials on the unit circle
$$
 \T=\{z \in \C \ : \ |z|=1\}.
$$
Any such construction requires that the underlying orthogonal scaling function satisfies
\begin{equation} \label{eq:refinement}
 \phi: \R \rightarrow \C^{K \times M}, \quad \phi=\sum_{\alpha \in \Z} \phi(2\cdot-\alpha) p_\alpha,
 \quad K,M \in \N, \quad K \le M,
\end{equation}
with the mask $\pb=\{p_\alpha \in \C^{M \times M} \ : \ \alpha  \in \{0,\ldots, n\} \}$, $n \in \N$.
The corresponding wavelet or multi-wavelet is defined by
$$
 \psi:\R \rightarrow \C^{K \times M}, \quad \psi=\sum_{\alpha \in \Z} \phi(2\cdot-\alpha) q_\alpha
$$
with the finitely supported mask $\qb=\{q_\alpha \in \C^{M \times M} \ : \ \alpha \in \{0,\ldots, n\}\}$, $n \in \N$,

By \cite{CollelaHeil, DL1992}, to ensure the existence of the compactly supported distributional solution
of (\ref{eq:refinement}),  we require that for the symbol
$$
  p(z)=\sum_{\alpha \in \Z} p_\alpha z^\alpha, \quad z \in \C \setminus \{0\},
$$
there exist $\tilde{M} \le M$ vectors $\bv_1,\ldots, \bv_{\tilde{M}} \in \R^M$ satisfying
\begin{equation} \label{sum_rules1}
 p(1) \bv_j=\bv_j \quad \hbox{and} \quad p(-1) \bv_j =0, \quad j=1, \ldots, \tilde{M}.
\end{equation}
Additionally, the other eigenvalues of $p(1)$ should be in the absolute values less than $1$.
In this case, we say that $\pb$ satisfies sum rules of  order $1$. Sum rules of order $1$ imply that
the associated multi-wavelet mask possesses a discrete vanishing moment
\begin{equation} \label{eq:van_moment}
q(1)^* \bv_j=0,\quad j=1, \ldots, \tilde{M},
\end{equation}
with the same vectors $\bv_j$ are as in \eqref{sum_rules1}.
 In the literature, the cases $\tilde{M}=1$ and $\tilde{M}=M$ are
called the $1$-rank and the full rank cases, respectively
\cite{CollelaHeil, Conti_Cotronei, MS}. The higher smoothness of $\phi$ imposes additional sum rule conditions
on the symbol $p(z)$, see e.g. \cite{Charina, Han, JetterPlonka}.

%
%
%
In the scalar or full rank cases (i.e. $\tilde{M}=M$), the sum rules of order $\ell+1$ are equivalent to
the existence of the factor $(1+z)^\ell$ in $p(z)$. The vanishing moment conditions
of order $\ell+1$ in these cases guarantee the existence of the factors $(1-z)^\ell$ in $q(z)$.

In the scalar case ($K=M=1$), the  wavelet construction by Daubechies \cite{Daub_ten_lectures} amounts
to defining the wavelet mask $\qb$ by
\begin{equation} \label{flip_trick}
 q_\alpha=(-1)^\alpha \, p_{n-\alpha}, \quad \alpha \in \{0, \ldots, n\}.
\end{equation}
In the case $M>1$, due to the non-commutativity of the matrices $p_\alpha$ and $q_\alpha$, the trick in \eqref{flip_trick}
does not apply.  Nevertheless, the interest in constructing multi-wavelets has not
decreased for the last $30$ years and it is motivated, for example, by the fact that the growth of the support of $\phi$ in this case is
decoupled from the smoothness of $\phi$ and symmetry does not conflict with orthogonality \cite{CollelaHeil}.

The constructions of the corresponding matrix-valued masks $\pb$ and $\qb$ are
based on the so-called QMF (quadrature mirror filter) and UEP (Unitary Extension Principle) conditions.
To state them, we define the matrix polynomial map
\begin{equation}\label{def:F}
 F: \C \rightarrow \C^{2M \times 2M}, \quad F(\xi)=\sum_{j=0}^N F_{j} \xi^{j}, \quad \xi=z^2,  \quad
 N=\lfloor \frac{n}{2} \rfloor,
\end{equation}
with the matrix coefficients $F_{j} \in  \C^{2M \times 2M}$, $M \in \N$,
$$
 F_{j}= \sqrt{2}\left(\begin{array}{cc}p_{2j} & q_{2j} \\ p_{2j+1} & q_{2j+1} \end{array}\right), \quad j=0,\ldots,N.
$$
The entries in the first column of $F$ are usually called the polyphase components of $p(z)$.
The QMF-condition states that
\begin{equation} \label{def:QMF}
 I-F^*(\xi) F(\xi)=0, \quad \xi \in \T,
\end{equation}
and, equivalently, the UEP-conditions are
\begin{equation} \label{def:UEP}
 I-F(\xi) F^*(\xi)=0, \quad \xi \in \T.
\end{equation}

To use classical results from the theory of linear systems, we look at $F$ in \eqref{def:F}
as a holomorphic function on the unit disk
$$
 \D=\{\xi \in \C \ : \ |\xi|<1\}.
$$
The QMF-condition \eqref{def:QMF} and the maximum principle imply that $F(\xi)$ is contractive for any
$\xi \in \D$.
Such matrix valued inner functions can be interpreted as transfer functions of conservative linear control systems; specifically it means that the representation
\begin{equation} \label{transfer_function}
 F(\xi)=A+B \xi (I-D\xi)^{-1} \, C, \quad  \xi \in \D,
\end{equation}
holds, with the $(2M+2MN)\times (2M+2MN)$ unitary matrix
\begin{equation} \label{def:partABCD}
 \left( \begin{array}{cc} A & B \\ C &D \end{array}\right)
\end{equation}
which we shall call the $ABCD$ matrix.
For possible further use in the present wavelet framework we refer for a proof and details to the mathematical article [5], which even treats realization theory in the case of several complex variables.


Note that the identity in \eqref{transfer_function} can be equivalently written as
\begin{eqnarray} \label{transfer_function_system}
\begin{array}{l}
  F(\xi)=A+\xi\, B  \, \ell_N(\xi),  \\ \noalign{\medskip}
 \ell_N(\xi)= C+ \xi \, D \, \ell_N(\xi),
 \end{array} \quad \xi \in \D.
\end{eqnarray}
This system of equations plays an important role in constructions of appropriate blocks of the $ABCD$ matrix.

\medskip The paper is organized as follows: In subsection \ref{sec:structure_lN} we discuss the structure of $\ell_N$ appearing in the transfer function system \eqref{transfer_function_system}. In subsection \ref{sec:structure_ABCD}, we provide the explicit form of the $ABCD$ matrix under the assumptions that $F$ satisfies
QMF-conditions on $\T$, see Theorem \ref{th:main}.
The compact support of the constructed multi-wavelets is ensured by the property $B\, D^N=0$, see
Proposition \ref{prop:BD}. The constructions of several compactly supported scaling functions and multi-wavelets are given in Section \ref{sec:examples}. In subsection \ref{sec:Potapov}, we compare our results with the characterization in
\cite{AlpayJorgensen}. The characterization in
\cite{AlpayJorgensen}  makes use of the so-called Potapov-Blaschke-products and is also valid for rational $F$.

We remark that although we prove our results for the case of dilation $2$, they all can be generalized in a natural way
 to the case of a general dilation factor, since this  mainly affects the dimensions of the matrices $F_j$ in (5).

\section{Characterization of orthogonal univariate multi-wavelets} \label{sec:main_uni}

The main goal of this section is to provide the explicit form of the $ABCD$ matrix  in \eqref{def:partABCD}
for all $F$ that satisfy  the QMF-condition \eqref{def:QMF} or, equivalently, the UEP-condition \eqref{def:UEP}.
We start by deriving the structure of $\ell_N$ in  \eqref{transfer_function_system}. Then we determine the
explicit structure of the $ABCD$ matrix (see Definition \ref{def:ABCD} and Theorem \ref{th:main}).
Further properties of the matrix $ABCD$ are studied in subsection \ref{sec:further_prop_ABCD}.

\subsection{Structure of $\ell_N$} \label{sec:structure_lN}

In order to  derive the structure of $\ell_N$ (see Theorem \ref{th:form_of_l1}), we  make use of the following straightforward observation.

\begin{Proposition} \label{prop:equivalent_form_QMF}
 The QMF-condition \eqref{def:QMF} is equivalent to the identity
 \begin{equation} \label{eq:aux}
 I-\left(\sum_{k=1}^{N} \left(\sum_{ {i,j \in \{0,\ldots,N\} } \atop -i+j=-k} F_{i}^* F_j \right) \bar{\eta}^k + \sum_{k=0}^{N} \left(\sum_{ {i,j \in \{0,\ldots,N\} } \atop -i+j=k} F_{i}^* F_j \right) \xi^k\right)=0
 \end{equation}
for $\xi, \eta \in \D$.

\end{Proposition}
\begin{proof}
 Note that the QMF-condition \eqref{def:QMF} is equivalent to
 \begin{equation} \label{eq:aux1}
  I-\sum_{k=-N}^{N} \left(\sum_{ {i,j \in \{0,\ldots,N\} } \atop -i+j=k} F_{i}^* F_j \right)
  \xi^k=0, \quad \xi \in \T,
 \end{equation}
 i.e. all the coefficients of the above Laurent-polynomial are equal to zero. This implies \eqref{eq:aux}
for all $\xi, \eta \in \D$. Conversely, if \eqref{eq:aux} is satisfied, then, setting $\bar{\eta}=\xi^{-1}$,
we obtain \eqref{eq:aux1} .
\end{proof}

The following result is an important step for determining the structure of the $ABCD$ matrix in
\eqref{def:partABCD}.

\begin{Theorem} \label{th:form_of_l1}
 The polynomial map $F$ satisfies the QMF-condition \eqref{def:QMF} if and only if
 \begin{equation} \label{eq:aux3}
  I-F^*(\eta) F(\xi)=(1- \bar{\eta}\xi)\ell_N^{*}(\eta) \ell_N(\xi), \quad \xi, \eta \in \D,
 \end{equation}
with
\begin{equation} \label{def:lN}
 \ell_N(\xi)=\left(\begin{array}{ccccc}F_1 & F_2& \dots & F_{N-1} & F_N \\ F_2 &F_3 & \dots &F_N & 0\\ \vdots & & & & \vdots \\ F_N & 0& \dots &0 & 0 \end{array}\right)
 \left(\begin{array}{c}1\\ \xi \\ \vdots \\ \xi^{N-1}\end{array}\right).
\end{equation}
\end{Theorem}

The proof of Theorem \ref{th:form_of_l1} for general $N \in \N$ is rather technical, thus, we first present the idea of the proof
on the example of the case $N=1$.

\begin{Example}
Let $\xi, \eta \in \D$. For $N=1$, due to \eqref{def:F}, we have
\begin{eqnarray} \label{eq:induction_start}
	I-F^*(\eta)F(\xi)&=&I-F_0^* F_0-F_0^*F_1\xi-F_1^*F_0\bar{\eta}-F_1^* F_1\bar{\eta}\xi \notag \\
	&=&
	I-\left(F_0^* F_0+F_1^* F_1\right)-F_0^*F_1\xi-F_1^*F_0\bar{\eta} \notag \\&+&
	(1-\bar{\eta}\xi)F_1^* F_1.
	\end{eqnarray}
The QMF-condition \eqref{def:QMF}, by Proposition \ref{prop:equivalent_form_QMF}, imply
$$
  I-\left(F_0^* F_0+F_1^* F_1\right)-F_0^*F_1\xi-F_1^*F_0\bar{\eta}=0,
$$
thus, yielding
$$
 I-F^*(\eta)F(\xi)=(1-\bar{\eta}\xi)\ell_1^*(\eta) \ell_1(\xi) \quad \hbox{with} \quad \ell_1(\xi)=F_1.
$$	
\end{Example}

\begin{proof}[Proof of Theorem \ref{th:form_of_l1}.] The proof of  ''$\Longrightarrow$``
is by induction on $N$. Let $\xi, \eta \in \D$.
The starting point of the inductive argument is given in \eqref{eq:induction_start}. For general
$N \in \N$, we need to show that
\begin{eqnarray*}
	I-F^*(\eta)F(\xi)&=&I-\left(\sum_{k=1}^{N} \left(\sum_{ {i,j \in \{0,\ldots,N\} } \atop -i+j=-k} F_{i}^* F_j \right) \bar{\eta}^k + \sum_{k=0}^{N} \left(\sum_{ {i,j \in \{0,\ldots,N\} } \atop -i+j=k} F_{i}^* F_j \right) \xi^k\right)\\ &&+(1-\bar{\eta} \xi) \ell_{N}^*(\eta)\ell_{N}(\xi).
\end{eqnarray*}
Then the QMF-condition, by Proposition \ref{prop:equivalent_form_QMF}, implies that
 \begin{eqnarray*}
	I-F^*(\eta)F(\xi)=(1-\bar{\eta} \xi) \ell_{N}^*(\eta)\ell_{N}(\xi).
\end{eqnarray*}
We start by writing
	\begin{eqnarray*}
		I-F^*(\eta)F(\xi)&=&I-\sum_{i=0}^{N-1} F^*_{i} \bar{\eta}^{i}\sum_{j=0}^{N-1} F_{j} \xi^{j}-F_N^* \bar{\eta}^N \sum_{j=0}^{N-1} F_{j} \xi^{j}\\&&-\sum_{i=0}^{N-1} F^*_{i} \bar{\eta}^{j}
		F_N \xi^N -F_N^* F_N \bar{\eta}^N \xi^N.
	\end{eqnarray*}
By the induction assumption
\begin{eqnarray*}
\lefteqn{I-\sum_{i=0}^{N-1} F^*_{i} \bar{\eta}^{i}\sum_{j=0}^{N-1} F_{j} \xi^{j}}\\
&=& I-\left(\sum_{k=1}^{N-1} \left(\sum_{ {i,j \in \{0,\ldots,N-1\} } \atop -i+j=-k} F_{i}^* F_j \right) \bar{\eta}^k + \sum_{k=0}^{N-1} \left(\sum_{ {i,j \in \{0,\ldots,N-1\} } \atop -i+j=k} F_{i}^* F_j \right) \xi^k\right) \\
&& +(1-\bar{\eta}\xi) \ell_{N-1}^*(\eta) \ell_{N-1}(\xi).
\end{eqnarray*}
Next observe that
\begin{eqnarray*}
	\lefteqn{-F_N^*\bar{\eta}^N \sum_{j=0}^{N-1} F_{j} \xi^{j}-\sum_{i=0}^{N-1} F^*_{i} \bar{\eta}^{j}
	F_N \xi^N - F_N^* F_N \bar{\eta}^N \xi^N} \\
	&&= -\sum_{j=0}^{N-1} \left(F^*_N F_j \bar{\eta}^{N}\xi^j + F^*_j F_N \bar{\eta}^{j}\xi^N \right) -F_N^*F_N \bar{\eta}^N \xi^N\\\\ &&= \sum_{j=1}^{N-1} (1-\bar{\eta}^j \xi^j)\left(F^*_N F_j \bar{\eta}^{N-j} + F^*_j F_N \xi^{N-j}\right)- \sum_{j=0}^{N-1} \left(F^*_N F_j \bar{\eta}^{N-j} + F^*_j F_N \xi^{N-j}\right)+\\
	&& (1-\bar{\eta}^N \xi^N)F_N^* F_N -F_N^*F_N.
\end{eqnarray*}
Thus, we get
\begin{eqnarray*}
	I-F^*(\eta)F(\xi)&=&I-\left(\sum_{k=1}^{N} \left(\sum_{ {i,j \in \{0,\ldots,N\} } \atop -i+j=-k} F_{i}^* F_j \right) \bar{\eta}^k + \sum_{k=0}^{N} \left(\sum_{ {i,j \in \{0,\ldots,N\} } \atop -i+j=k} F_{i}^* F_j \right) \xi^k\right)\\ &&+(1-\bar{\eta} \xi) \ell_{N-1}^*(\eta)\ell_{N-1}(\xi) \\
	&&+ \sum_{j=1}^{N-1} (1-\bar{\eta}^j \xi^j)\left(F^*_N F_j \bar{\eta}^{N-j} + F^*_j F_N \xi^{N-j}\right)+(1-\bar{\eta}^N \xi^N)F_N^* F_N.
\end{eqnarray*}
The QMF-condition, due to Proposition \ref{prop:equivalent_form_QMF}, and the geometric sum
argument for $(1-\bar{\eta}^j \xi^j)$ lead to
\begin{eqnarray*}
	I-F^*(\eta)F(\xi)&=&(1-\bar{\eta} \xi) \Big( \ell_{N-1}^*(\eta)\ell_{N-1}(\xi) \\&&+
	\sum_{j=1}^{N-1} \left(\sum_{k=0}^{j-1} \left(\bar{\eta} \xi\right)^{j-1-k}\right) \left(F^*_N F_j \bar{\eta}^{N-j} + F^*_j F_N \xi^{N-j}\right) \\
	&& + \sum_{k=0}^{N-1} \left(\bar{\eta} \xi\right)^{N-1-k}F^*_N F_N \Big).
\end{eqnarray*}
By the definition of $\ell_{N-1}$, we have
$$
  \ell_{N-1}^*(\eta)\ell_{N-1}(\xi)=\sum_{i=1}^{N-1} \left(  1  \ \dots \ \bar{\eta}^{N-1-i}\right) \left( \begin{array}{c} F_i^* \\ \vdots \\ F_{N-1}^*\end{array}\right) \left( \begin{array}{ccc} F_i & \dots & F_{N-1}\end{array}\right)
 \left( \begin{array}{c} 1 \\ \vdots \\ \xi^{N-1-i}\end{array}\right).
$$
Note that, for $i=1, \ldots, N-1$,
\begin{eqnarray*}
\lefteqn{\left( \begin{array}{ccc} 1 & \dots& \bar{\eta}^{N-i}\end{array}\right) \left( \begin{array}{c} F_i^* \\ \vdots \\ F_{N}^*\end{array}\right) \left( \begin{array}{ccc} F_i & \dots & F_{N}\end{array}\right)
\left( \begin{array}{c} 1 \\ \vdots \\ \xi^{N-i}\end{array}\right)}\\
&=&\left( \begin{array}{ccc} 1 & \dots& \bar{\eta}^{N-i}\end{array}\right) \left( \begin{array}{c} F_i^* \\ \vdots \\ F_{N-1}^* \\ 0\end{array}\right) \left( \begin{array}{cccc} F_i & \dots & F_{N-1} & 0\end{array}\right)
\left( \begin{array}{c} 1 \\ \vdots \\ \xi^{N-i}\end{array}\right)\\
&& 	+\sum_{j=i}^{N-1} \left(\bar{\eta} \xi\right)^{j-i} \left(F^*_N F_j \bar{\eta}^{N-j} + F^*_j F_N \xi^{N-j}\right)+\left(\bar{\eta} \xi \right)^{N-i} F_N^* F_N
\end{eqnarray*}
and the reordering of the summands leads to
\begin{eqnarray*}
 \lefteqn{\sum_{j=1}^{N-1} \left(\sum_{k=0}^{j-1} \left(\bar{\eta} \xi\right)^{j-1-k}\right) \left(F^*_N F_j \bar{\eta}^{N-j} + F^*_j F_N \xi^{N-j}\right)}  \\
 &=&\sum_{i=1}^{N-1} \left(\sum_{j=i}^{N-1} \left(\bar{\eta} \xi\right)^{j-i}\right) \left(F^*_N F_j \bar{\eta}^{N-j} + F^*_j F_N \xi^{N-j}\right).
\end{eqnarray*}
Therefore, we obtain  \eqref{eq:aux3}
with $\ell_N$ in \eqref{def:lN}. The proof of ''$\Longleftarrow$`` follows by substituting
$\eta = \xi$ in  \eqref{eq:aux3} and since $\bar \xi=\xi^{-1}$ for $\xi \in {\mathbb T}$.
\end{proof}

\subsection{Structure of the $ABCD$ matrix} \label{sec:structure_ABCD}
The main result of this section characterizes all orthogonal wavelets and multi-wavelets in terms of
transfer function representations for the analytic map $F:\D \rightarrow \C^{2M \times 2M}$ defined in \eqref{def:F}. Such representations involve certain complex matrices,
which we define next.

\begin{Definition} \label{def:ABCD}
For $F_j \in \C^{2M \times 2M}$, $j=0, \ldots, N$, in \eqref{def:F}, define
the $2M(N+1) \times 2M(N+1)$ block matrix
\begin{eqnarray*}
\left(\begin{array}{ccc} A &\vline& B \\ \hline C& \vline &D\end{array} \right)=
\left(\begin{array}{cccccc} F_0 &\vline& F_N & \dots &\dots &F_1  \\ \hline
F_1 & \vline & F_0 & F_N &  & \vdots \\ \vdots & \vline & F_1 & \ddots & \ddots & \vdots\\
\vdots & \vline &  & \ddots & \ddots &F_N \\
F_{N} & \vline & F_{N-1} & \dots &F_1 & F_0
\end{array} \right) \cdot \left(\begin{array}{ccc} I &\vline& 0 \\ \hline 0& \vline &U\end{array} \right),
\end{eqnarray*}
where the $2MN \times 2MN$ matrix $U$ is given by
\begin{eqnarray*}
U&=&\left(\begin{array}{cccc} F_0^*+F_N^* & F_1^* &  \dots & F_{N-1}^*\\
F_{N-1}^* & \ddots & \ddots & \vdots \\
\vdots & \ddots & \ddots &F_1^*  \\ F_1^* & \dots & F_{N-1}^* & F_0^*+F_N^*\end{array}\right).
\end{eqnarray*}
\end{Definition}
The proof of our main result, Theorem \ref{th:main}, relies on the unitary property of the matrix $U$
in Definition \ref{def:ABCD}.

\begin{Example} In the case $N=1$, it is easy to check that the QMF-condition  \eqref{def:QMF} implies that the matrix $U$ in Definition \ref{def:ABCD} is unitary
$$
 UU^*=(F_0^*+F_1^*)(F_0+F_1)=F_0^*F_0+F_1^*F_1=I,
$$
where we used that $F_1^*F_0=F_0^*F_1=0$.
The case $N=2$ illustrates the idea of the proof  of the unitary property of $U$ in the general case (see
Proposition \ref{prop:U_unitary}). Assume that the QMF-condition is satisfied. Let $I_{2M}$ be the $2M \times 2M$ identity matrix. Then writing (using the circulant structure of $U$)
\begin{eqnarray*}
U=\left(\begin{array}{cc}F_0^*+F_2^*&0\\0&F_0^*+F_2^*\end{array}\right)
+\left(\begin{array}{cc}F_1^*&0\\0&F_1^* \end{array}\right) \left(\begin{array}{cc}0&I_{2M}\\I_{2M}&0 \end{array}\right)
\end{eqnarray*}
we get, using $F_2^*F_0=F_0^*F_2=0$,
\begin{eqnarray} \label{eq:unitaryUN_2_2}
 UU^*&=&\left(\begin{array}{cc}F_0^*F_0+F_2^*F_2&0\\0&F_0^*F_0+F_2^*F_2\end{array}\right) \notag \\&+&
 \left(\begin{array}{cc} 0&F_1^*F_2+F_2^*F_1\\F_1^*F_2+F_2^*F_1&0\end{array}\right)+
 \left(\begin{array}{cc}F_1^*F_1&0\\0&F_1^*F_1\end{array}\right).
\end{eqnarray}
Thus, the rest of the QMF-conditions imply that $UU^*=I_{4M}$.
\end{Example}


\begin{Proposition} \label{prop:U_unitary}
If the polynomial map $F$ in \eqref{def:F} satisfies the QMF-condition  \eqref{def:QMF}, then the matrix $U$
from Definition \ref{def:ABCD} is unitary.
\end{Proposition}
 \begin{proof} For the $2M \times 2M$ identity matrix $I_{2M}$, define
$$
 P:=\left(\begin{array}{ccccc} 0 & \dots & 0 & I_{2M} \\ I_{2M} & \dots & 0 & 0 \\ \vdots & \ddots & \vdots &  \vdots\\
0 & \dots & I_{2M} & 0\end{array} \right) \in \R^{2MN \times 2MN}.$$
Note that, similarly to the standard definition of circulant matrices,
$$
 U=(I_N \otimes (F_0^*+F_N^*)) P^0+\sum_{j=1}^{N-1} (I_N \otimes F^*_{N-j})\, P^j.
$$
where $P^0$ is the $2MN \times 2MN$ identity matrix. Analogously,
$$
 U^*=(I_N \otimes (F_0+F_N))P^0+\displaystyle{\sum_{j=1}^{N-1} (I_N \otimes F_{j}})\, P^j.
$$
Using these representations of $U$ and of $U^*$ and the fact that $P^{N+k}=P^k$, $k=0,\dots,N-1$, similarly
to \eqref{eq:unitaryUN_2_2}, we get that the product $UU^*$ contains
$\displaystyle \sum_{j=0}^N F_j^*F_j=I_{2M}$ on its main diagonal and other (zero) QMF-conditions on its subdiagonals.
Thus, the claim follows.
\end{proof}

We are finally ready to state the following characterization of all compactly
supported orthogonal wavelet and multi-wavelet masks.

\begin{Theorem} \label{th:main}
Let $F$ be a polynomial map in \eqref{def:F}. The map $F$ satisfies the QMF-condition
 \eqref{def:QMF}  if and only if $F$ satisfies
 \begin{equation}\label{eq:ABCDrepresentation_of_F}
 F(\xi)=A+B\xi(I-D\xi)^{-1} C, \quad \xi \in \D,
 \end{equation}
with the unitary map
\begin{equation} \label{eq:ABCD}
 \left( \begin{array}{cc} A & B \\ C &D \end{array}\right) : \begin{array}{lll} \C^{2M} & & \C^{2M} \\ \oplus & \rightarrow & \oplus
 \\ \C^{2MN} & & \C^{2MN}
 \end{array}
\end{equation}
given in Definition \ref{def:ABCD}.
\end{Theorem}
\begin{proof}
 The proof of ''$\Longrightarrow$`` consists of two parts. Firstly, note that the
special choice of the matrices $A$, $B$, $C$, $D$, Proposition \ref{prop:U_unitary}  and the
hypothesis imply that the matrix   in \eqref{eq:ABCD} is indeed unitary.
Next, we show that $F$ satisfies \eqref{eq:ABCDrepresentation_of_F}. Let $\xi \in \D$. By \cite{BT}, the identity in
\eqref{eq:ABCDrepresentation_of_F} is equivalent to
the system of equations
\begin{equation} \label{system_for_ABCD_1}
\begin{array}{l}
 A+\xi\, B  \, \ell_N(\xi)=F(\xi) \\ \noalign{\medskip}
 C+ \xi \, D \, \ell_N(\xi) = \ell_N(\xi),
 \end{array}
\end{equation}
with $\ell_N$ as in Theorem \ref{th:form_of_l1}. By the definitions of the matrices $A$, $C$ and the polynomial
map $\ell_N$, the system in \eqref{system_for_ABCD_1} is equivalent to
\begin{eqnarray} \label{system_for_ABCD_2}
 \begin{array}{l}
 \xi\, B  \, \ell_N(\xi) =F_1 \xi + \ldots F_N \xi^N= \left( \begin{array}{ccc} F_1 & \dots F_N \end{array}\right) \left( \begin{array}{c} \xi \\ \vdots \\ \xi^N \end{array}\right)   \\\noalign{\medskip}
 \xi \, D \, \ell_N(\xi)  = \ell_N(\xi)- C = \left( \begin{array}{cccc} F_2 & \dots & F_N & 0
 \\ \vdots  & & & \vdots\\ F_N & & & \\ 0 & \dots &0&0  \end{array}\right) \left( \begin{array}{c} \xi \\ \vdots \\ \xi^N \end{array}\right).
 \end{array}
\end{eqnarray}
After the division of both sides of \eqref{system_for_ABCD_2} by $\xi$ and
by  $U^*U=I$, we get another equivalent system
\begin{eqnarray} \label{system_for_ABCD_3}
\begin{array}{l}
	 \Big( B -
	\left( \begin{array}{cccc} F_0+F_N & F_{N-1} & \dots & F_1 \end{array}\right) U\Big) \,
	\ell_N(\xi)=0  \\	\noalign{\medskip}
	 \Big( D -
	\left( \begin{array}{cccc} F_1& F_0+F_N &  \dots & F_2 \\
		\vdots & \vdots& \ddots & \vdots \\ F_{N-1} & F_{N-2} & \dots & F_0+F_N \\
		0 & 0 & \dots & 0 \end{array}\right) U\Big) \,
	\ell_N(\xi)=0.
	\end{array}
\end{eqnarray}
Observe that the QMF-conditions yield
\begin{eqnarray} \label{system_for_ABCD_4}
	U\ell_N(\xi) &=& \Big( \left( \begin{array}{ccc} F_N^*&   \dots & 0\\
		\vdots  & \ddots & \vdots \\ F_{1}^*  & \dots & F_N^*\end{array}\right) +
	\left( \begin{array}{ccc} F_0^*&   \dots & F_{N-1}^*  \\
		\vdots &  \ddots & \vdots \\ 0 &  \dots & F_0^*  \end{array}\right)\Big)
	\ell_N(\xi) \notag \\&=&
	\left( \begin{array}{ccc} F_N^*&   \dots & 0\\
		\vdots  & \ddots & \vdots \\ F_{1}^*  & \dots & F_N^*\end{array}\right) \ell_N(\xi)	
\end{eqnarray}
Thus,
due to $F_0F_N^*=0$ and \eqref{system_for_ABCD_4}, the first identity in \eqref{system_for_ABCD_3} is
satisfied for $B$ in Definition \ref{def:ABCD}
\begin{eqnarray*} \label{system_for_ABCD_5}
\lefteqn{- \left( \begin{array}{cccc} F_0 & 0 & \dots & 0 \end{array}\right) U \,
\ell_N(\xi)}\\
&=&  - \left( \begin{array}{cccc} F_0 & 0 & \dots & 0 \end{array}\right)
\left( \begin{array}{ccc} F_N^*&   \dots & 0\\
	\vdots  & \ddots & \vdots \\ F_{1}^*  & \dots & F_N^*\end{array}\right) \,
\ell_N(\xi)=0.\\
\end{eqnarray*}
The rest of the QMF- and UEP-conditions and \eqref{system_for_ABCD_4} imply that the second identity in \eqref{system_for_ABCD_3} is satisfied for $D$ in Definition \ref{def:ABCD}
\begin{eqnarray*} \label{system_for_ABCD_6}
&&  \Big( D -
\left( \begin{array}{cccc} F_1& F_0+F_N &  \dots & F_2 \\
	\vdots & \vdots& \ddots & \vdots \\ F_{N-1} & F_{N-2} & \dots & F_0+F_N \\
	0 & 0 & \dots & 0 \end{array}\right) U\Big) \,
\ell_N(\xi) =\\ &&\Big( \left( \begin{array}{cccc} F_0& 0&  \dots & 0\\
	F_1&F_0 &\dots&0 \\
\vdots & \vdots & \ddots & \vdots \\ F_{N-1}  &F_{N-2}& \dots & F_0\end{array}\right) -
\left( \begin{array}{cccc} F_1& F_0&  \dots & 0 \\
\vdots & \vdots& \ddots & \vdots \\ F_{N-1} & F_{N-2} & \dots & F_0 \\
0 & 0 & \dots & 0 \end{array}\right)\Big)U \,
\ell_N(\xi)=0.
\end{eqnarray*}

The proof of  ''$\Longleftarrow$`` follows by a linear algebra argument. Namely, \eqref{eq:ABCDrepresentation_of_F},
in its equivalent form
$$
 \left( \begin{array}{c} F(\xi) \\ \ell_N(\xi)\end{array} \right)= \left( \begin{array}{cc} A&B \\ C&D\end{array} \right)  \left( \begin{array}{c} I \\ \xi \ell_N(\xi)\end{array} \right), \quad \ell_N(\xi) = (I-\xi D)^{-1}C, \quad \xi \in \D,
$$
and the fact that $ABCD$ is a unitary matrix is reflected in the conservation law
\begin{equation} \label{eq:conservation_law}
 \|F(\xi)\|^2 + \|\ell_N(\xi)\|^2 =1 + \| \xi \ell_N(\xi)\|^2, \quad \xi \in \D.
\end{equation}
Note that the matrix $D$ is contractive, so $\ell_N(\xi)$ is a rational function,
analytic in the unit disk $\D$. Let $z$ be a point on the unit torus $\T$ which is not a pole of $\ell_N$. Passing to the limit $\xi \rightarrow z$ in the identity \eqref{eq:conservation_law}, we obtain
$$
1 = \| F(z)\|^2 \quad  \hbox{for every $z \in \T$ which is not a pole of $\ell_N$}.
$$
Recall that $F$ is assumed to be a polynomial map, hence
$$
1 = \| F(z)\|^2 \quad \hbox{for all} \quad z \in \T,
$$
i.e. the matrix $F(z)$ is unitary for all $z \in \T$.
\end{proof}

\subsection{Further properties of the $ABCD$ matrix} \label{sec:further_prop_ABCD}

In this section, we analyze the properties of the matrices $B$ and $D$ that guarantee that the
representation in \eqref{eq:ABCDrepresentation_of_F} leads to a polynomial $F$.

\begin{Proposition} \label{prop:BD}
If $F$ of degree $N$ in \eqref{def:F} satisfies QMF-condition \eqref{def:QMF},
then the matrices $B$ and $D$ from Definition \ref{def:ABCD} satisfy
$$
 B \, D^N=0.
$$
\end{Proposition}
\begin{proof}	
 Using Definition \ref{def:ABCD}, we write $B=\tilde{B} \, U$ and $D=\tilde{D} \, U$. Note that, due to the invertibility of $U$, we only need to show that $\tilde{B} ( U \, \tilde{D})^N=0$.
To prove the claim, we show that
\begin{eqnarray*}
&&  \tilde{B} \, U  \, \tilde{D}=(F_N  \ \dots \ F_{2} \   F_{1} ) \, U \, \tilde{D}=(0 \ F_N  \ \dots \  F_{2}), \\
&&  \tilde{B} \, \left( U \, \tilde{D}\right)^2
= (0 \ F_N  \ \dots \ F_{2}) \, U \, \tilde  D=(0 \ 0 \ F_N  \ \dots \  F_{3})
\end{eqnarray*}
and so on until
$$
 \tilde{B} \, \left( U \, \tilde{D}\right)^N = (0 \ \ldots \ 0 \  \ F_N) \, U \, \tilde  D=(0 \ \ldots \ \ 0).
$$
First, note that, due to the structure of $U$ and $\tilde{D}$, we have
$$
 U = U_1+U_2:=\left( \begin{array}{cccc} F_0^* & F_1^*&  \dots & F_{N-1}^* \\
 	\vdots & \ddots& \ddots & \vdots \\ 0 &  & \ddots & F_1^* \\
 	0 & 0 & \dots & F_0^* \end{array}\right)+
   \left( \begin{array}{cccc} F_N^*&  \dots &  0 & 0\\ F_{N-1}^*& \ddots & & 0 \\
 	\vdots & \ddots & \ddots & \vdots \\ F_{1}^*  & \dots& F_{N-1}^* & F_N^*\end{array}\right)
 	$$
 and
 	$$
 \tilde{D}=\tilde{D}_1+\tilde{D}_2:= \left( \begin{array}{cccc} F_0&  \dots &  0 & 0\\ F_1& \ddots & & 0 \\
 	\vdots & \ddots & \ddots & \vdots \\ F_{N-1}  & \dots& F_{1} & F_0\end{array}\right)+
   \left( \begin{array}{cccc} 0 & F_N&  \dots & F_2 \\
 	\vdots & \ddots& \ddots & \vdots \\ 0 &  & \ddots & F_N \\
 	0 & 0 & \dots & 0 \end{array}\right).
$$	
For $\ell = N-1, \ldots, 0$,
by QMF-condition, we have $U_2 \tilde{D}_1=0$ and, by UEP-condition, $(0_{N-\ell-1} \  F_N  \ \dots \  F_{N-\ell} ) U_1=0$. Thus,
\begin{eqnarray*}
 	(0_{N-\ell-1} \  F_N  \ \dots \  F_{N-\ell} )\,  U  \tilde{D}&=&(0_{N-\ell-1} \  F_N  \ \dots \  F_{N-\ell}) \, U_2 \tilde{D}_2 \\
 	&=& \Big( \sum_{(i,j) \in \Gamma_{\ell,1,N}} F_i\, F_j^* \ \ldots \  \sum_{(i,j) \in \Gamma_{\ell,N,N}} F_i\, F_j^*\Big) \, \tilde{D}_2,
\end{eqnarray*}
where, for $k=1, \ldots, N$,
$$
 \Gamma_{\ell,k,N}=\{(i,j) \ : \ -i+j=\ell-N+k, \ i\in\{N-\ell, \ldots, N\}, \ j \in \{1, \ldots, N\}\}.
$$
The UEP-condition implies, for $k=1, \ldots, N$,
$$
 \sum_{(i,j) \in  \Gamma_{\ell,k,N}} F_i\, F_j^* =\left\{\begin{array}{cc}
  I- \displaystyle \sum_{(i,j) \in  \Lambda_{\ell,k,N}} F_i\, F_j^*, &  \hbox{if} \ \ k=N-\ell, \\ \\
   -\displaystyle \sum_{(i,j) \in \Lambda_{\ell,k,N}} F_i\, F_j^*, & \hbox{otherwise},
   \end{array}\right.
$$
where
$$
  \Lambda_{\ell,k,N}=\{(i,j) \ : \ -i+j=\ell-N+k, \ i\in\{N-\ell, \ldots, N\}, \ j \in \{0, \ldots, k-1\}\}.
$$
Therefore,
\begin{eqnarray*}
	\Big( \sum_{(i,j) \in \Gamma_{\ell,1,N}} F_i\, F_j^* \ \ldots \  \sum_{(i,j) \in \Gamma_{\ell,N,N}} F_i\, F_j^* \Big) &=&
	 (0_{N-\ell-1} \ I \ 0_{\ell}) \\ &-&\sum_{ N-\ell-k \ge 0 \atop k \in \{1,\ldots,N\} } \, F_{N-k-\ell}\, (0_{k-1}\ F_0^*\  \ldots \ F_{N-k}^*).
\end{eqnarray*}
Multiplication by $\tilde{D}_2$ of both sides of the above equation, due to QMF-condition, yields the claim.
\end{proof}

\section{Special case $N=1$} \label{sec:Potapov}

In this section, we consider the special situation of polynomials $F$ of degree $N=1$.   The following Lemma \ref{lem:N_1} is crucial for
comparison of Theorem \ref{th:main} with \cite[Theorem 3.1]{AlpayJorgensen} and also for our specific constructions in Section \ref{sec:examples}.

\begin{Lemma} \label{lem:N_1}
	Let $A,B,C,D$ be matrices in $\R^{2M\times 2M}$. The following  two sets of conditions  $(I)$ and $(II)$ are equivalent.
	
	\medskip \noindent
	
	\begin{description}
		\item[$(I)$]
		$\left\{
		\begin{array}{lll}
		&\hbox{The block matrix} \
		\left( \begin{array}{cc} A & B \\
		C &D \end{array} \right) \hbox{is unitary\ }, & (I.a)  \\
		\\
		&DC=0,\quad BC=C, &(I.b)\\
		\\
		&B=CU,\quad D=AU,\quad U=A^*+C^*. & (I.c)
		\end{array}\right.
		$
		
		\bigskip
		\bigskip
		\item[$(II)$]
		$\left\{
		\begin{array}{lll}
		&B^2=B,\quad B^*=B. & (II.a)  \\
		\\
		& D=I-B;&(II.b)\\
		\\
		&C=BU^*, \quad A=DU^*,\quad U=A^*+C^*.& (II.c)
		\end{array}\right.
		$
	\end{description}
\end{Lemma}
\begin{proof}
Note first that the conditions $(I.c)$ and $(II.c)$ are equivalent.

\noindent Assume that $(I)$ are satisfied. The $ABCD$ matrix is unitary, in
particular $A^* A +C^*C=I$ and $A^*B+C^*D=0$, and
$(II.c)$ imply that $U$ is unitary
$$
 UU^*=(A^*+C^*)(A+C)=A^*A+C^*C+A^*C+C^*A=I+(A^*B+C^*D)U^*=I.
$$
Next, $BC=C$ and $(II.c)$ yield
	$$
	B^2=B\, B=B\, (CU)=(BC) \, U=C\, U=B.
	$$
	The definitions of $A$ and $C$ in $(II.c)$ imply $AC^*=\frac{1}{2}AC^* + \frac{1}{2}DB^*$.
	Because $ABCD$ is unitary, in particular $AC^*+DB^*=0$, we get
	$$
	B^*=UC^*=(A+C)C^*=AC^*+CC^*=CC^*=B.
	$$
	By $(II.c)$ we obtain $U^*=A+C=(D+B)U^*$. Thus, $U^*U=I$,  leads to $D=I-B$.
	
\vspace{0.2cm}
	
\noindent Assume that $(II)$ are satisfied.	By $(II.a)$ and $(II.b)$, the matrix $D$ also satisfies
$D^2=D$ and $D^*=D$. Thus, by the definitions of $A$ and $C$ in $(II.c)$, $U=A^*+C^*$ is unitary
$$
 U^*U=D^2+B^2+DB+BD=(I-B)^2+B^2+(I-B)B+B(I-B)=I.
$$
Moreover, $(I.a)-(I.c)$ yield	
$$
A^* A +C^*C=U(D^*D+B^*B)U^*=U((I-B)^2+B^2)U^*=UU^*=I,
$$
which also proves that $B^{*} B +D^*D=I$. Furthermore,
$$
 A^*B +C^*D=UD^*B+UB^*D=U((I-B)B+B(I-B))=0.
$$
Similarly, $B^*A+D^*C=0$. Next, we show that  $DC=0$ and $BC=C$. From $(II.a)$ and $(II.b)$ we get
$$
	DC=(I-B)C=(I-B)BU^*=(B-B^2)U^{*}=0
$$
and, by the definition of $C$,
$$
	BC=BBU^{*}=B^{2}U^{*}=BU^{*}=C,
$$
which concludes the proof. \end{proof}

\section{Examples} \label{sec:examples}

This section  illustrates our results with several examples. In particular, for $N=1$, the examples point out the strength of the algorithm given by the conditions $(II)$
in Lemma \ref{lem:N_1}. This algorithm allows us to characterize all possible wavelet and multi-wavelet masks with support
on $[0,2]$ or $[0,3]$. In subsection \ref{sec:examples_N_2}, we show how to apply the result of Theorem \ref{th:main} for
construction of $F$ in \eqref{def:F} of degree $N=2$ with support on $[0,5]$.

\subsection{Wavelets and multi-wavelets supported on $[0,2]$ or on $[0,3]$} \label{sec:examples_N_1}

Several properties of $F$ in \eqref{def:F} are similar in the scalar ($K=M=1$) and full
rank ($K=M>1$) cases. The corresponding masks are characterized in subsection
\ref{sec:examples_N_1_full_rank}. The rank one case ($1=K<M$) is considered in subsection
\ref{sec:examples_N_1_rank_1}.

\subsubsection{Wavelets and full rank multi-wavelets} \label{sec:examples_N_1_full_rank}

We first consider the full rank $K=M$ matrix case, which includes the wavelet case $K=M=1$. Note first that the full rank
requirement in the case $K=M$ uniquely determines the unitary matrix $U$. In fact, since  $U=F_0^*+F_1^*$, the
first order sum rule/vanishing moments conditions \eqref{sum_rules1}-\eqref{eq:van_moment} are equivalent to
$$\left(\begin{array}{rr}
I_M & I_M \\
I_M & -I_M
\end{array}\right)
U^*=\sqrt{2}\, \left(\begin{array}{cc}
I_M & 0 \\
0 & W
\end{array}\right), \quad W^*W=I_M.
$$
Therefore,
$$
 U=\frac{\sqrt{2}}{2} \left(\begin{array}{rr}
   I_M & W \\
   I_M & -W
 \end{array}\right).
$$
By Lemma \ref{lem:N_1}, the masks $\pb$ and $\qb$ are, thus, determined by the choice of the projection $B$.
In the scalar case $K=M=1$, there are only three choices for $B$: identity or two one-parameter families
\begin{equation}
\label{eq:Bscalar}
 B_+=\left(\begin{array}{cc} b & \sqrt{b-b^2} \\  \sqrt{b-b^2} & 1-b
\end{array}\right) \quad \hbox{and} \quad   B_{-}=\left(\begin{array}{cc} b & - \sqrt{b-b^2} \\ - \sqrt{b-b^2} & 1-b
\end{array}\right), \quad b \in \R.
\end{equation}
Once a particular $B$ is chosen, set $F_0=A=(I-B)U^*$, $F_1=C=BU^*$. To recover Haar masks $\pb$ and $\qb$ choose $B=I_2$
or $b=1$.  To impose additional sum rules/vanishing moments \cite{Charina, Han, JetterPlonka}
for $k=1$
we solve for $b$
\begin{equation} \label{eq:additionalmoments}
\left(\begin{array}{rrrr}
0 & -1 & 2 & -3 \\
0 & 1 & 2 & 3
\end{array}\right)
\left(\begin{array}{c}
F_0\\F_1
\end{array}\right)=
\left(\begin{array}{c}
0\\0\\0\\0
\end{array}\right).
\end{equation}
This system yields a unique solution $b=\frac{3}{4}$, determining the Daubechies (D4) masks $\pb$ and $\qb$ with the supports
$[0, 3]$.

In the case $K=M=2$, there are several choices for the projection $B$.  If we look for the masks $\pb$ and $\qb$ with supports
$[0, 2]$, then the only possible projections $B$ are given by
\begin{equation}\label{eq:Bmatrix}
B=\left(\begin{array}{cc}
B_{\pm} & 0_2 \\
0_2 & 0_2
\end{array}\right),
\end{equation}
where the blocks $B_{\pm}$ are given in \eqref{eq:Bscalar} and $0_2$ are $2 \times 2$ zero blocks.
Note, however, that these choices of $B$ lead to \emph{essentially diagonal} matrix-valued masks $\pb$ and $\qb$ specified
in \cite{CCSAvignone}. Such essentially diagonal matrix-valued masks $\pb$ and $\qb$ are equivalent to some scalar masks $\pb_j$
and $\qb_j$, $j=1,2$, since $\pb$ and $\qb$ are jointly diagonalizable.
This means that there are only trivial full rank matrix-valued masks $\pb$ and $\qb$ supported on $[0,2]$.

If we consider $K=M=2$ and  look for the masks with supports $[0,3]$, then we retrieve e.g. all the full rank
families of masks $\pb$ and $\qb$ in \cite{CH}. For example, the ones in \cite[Table A.4]{CH} are obtained for the projections
$$
 B=\left( \begin {array}{cccc}  1-b_1 &0&0&\sqrt{b_1-b_1^2}
 \\ 0&1-b_2&\sqrt{b_2-b_2^2}&0
 \\ 0&\sqrt{b_2-b_2^2}&b_2&0
 \\ \sqrt{b_1-b_1^2}&0&0&b_1\end {array}
 \right),  \quad 0\le b_1,b_2\le 1.
$$
Whereas, the masks $\pb$ and $\qb$  in \cite[Table A.3]{CH} come from the projection
$$
\small{
 B= \left( \begin {array}{cccc} 1&0&0&0\\ \noalign{\medskip}0&\frac 14\,
 \left( b+1 \right)^{2}&\frac 14\,({b}^{2}-1) &-\frac {\sqrt 2}4\sqrt {-{b}^{2}+1}
 (b+1 ) \\
 0&\frac 14\,({b}^{2}-1)&\frac 14\,
\left( b-1 \right) ^{2}&-\frac {\sqrt2}4\,\sqrt {-{b}^{2}+1} \left( b-1\right) \\ 0&-\frac {\sqrt 2}4\sqrt {-{b}^{2}+1} ( b+1
) &-\frac {\sqrt 2}4\sqrt {-{b}^{2}+1} ( b-1 ) &\frac 12\,(-{b}^{2}+1)\end {array} \right),
}$$
with  $ |b|\le 1$.
As in the scalar case, the free parameters $b$, $b_1$ and $b_2$ are determined by imposing additional sum rules/vanishing moment conditions.
These conditions are similar to the ones in \eqref{eq:additionalmoments}, due to the nature of the full rank case:
$$
 \left(\begin{array}{rrrr} 0 & -1 & 2 & -3 \\ 0 & 1 & 2 & 3
 \end{array}\right)
 \quad \hbox{replaced by} \quad \left(\begin{array}{rrrr}
0_2 & -I_2 & 2I_2 & -3I_2 \\
0_2 & I_2 & 2I_2 & 3I_2
\end{array}\right).
$$

\subsubsection{$1$-rank orthogonal multi-wavelets} \label{sec:examples_N_1_rank_1}

In this subsection we relax the full rank requirement and consider the multi-wavelet (rank $1$) setting with $1=K<M=2$.
If we require the support of the masks $\pb$ and $\qb$ to be $[0, 2]$, then the projection $B$ is as in \eqref{eq:Bmatrix}.
To impose sum rule/vanishing moment conditions on the unitary matrix $U$, we consider, for some non-zero
$v=(v_1,v_2)^T \in \R^2$, the system
 \begin{equation}\label{eq:y}
 \left(\begin{array}{cccc} v_1 & v_2 & 0 & 0 \end{array}\right)\, U=
 \frac {\sqrt{2}}{2}\left(\begin{array}{cccc} v_1 & v_2 & v_1 & v_2 \end{array}\right).
 \end{equation}
 and
 \begin{equation}\label{eq:y1}
 \left(\begin{array}{cccc} v_1 & v_2 & v_1 & v_2 \end{array}\right)\, U^*=
  {\sqrt{2}}\left(\begin{array}{cccc} v_1 & v_2 & 0 & 0 \end{array}\right).
 \end{equation}
Note that, by \cite{Charina, Han}, we can restrict our attention w.l.g. to the case
$v=(1,0)^T$ (though one can allow for different $v \in \R^2$ to be able to reproduce other known constructions).
Since $v=\hat \phi(0)$, this happens for example under the assumption that the components of $\phi=(\phi_1, \phi_2)$ are
symmetric/antisymmetric, respectively, around the center of their support, see e.g. \cite{ChuiLian}.
In this case, the first row of $U$ can be  determined from
$$\left(\begin{array}{cccc} 1 & 0 & 0 & 0\end{array}\right)\, U=
\frac {\sqrt{2}}{2}\left(\begin{array}{cccc} 1 & 0 & 1 & 0\end{array}\right).
$$
To impose the symmetry/antisymmetry assumptions, we set the zero entry of the mask $\pb$ to be $p_0=S\, p_2\, S$ and its
first entry $p_1$ to be diagonal. Here we use
$S:=\left(\begin{array}{cc} 1 & 0\\ 0 & -1  \end{array}\right)$. Then $F_0$ and $F_1$ are diagonal matrices and
the matrix $U$, which depends only on one parameter, is one of the following matrices
$$V_1\left( \begin{array}{cc}
U_1 & U_2\\
-U_2 & U_1
\end{array}\right) V_2,
$$
where
$$
 U_1=\left(\begin{array}{cc} \frac {\sqrt{2}}2 & 0 \\ 0 & \ell \end{array}\right), \quad
 U_2=\left(\begin{array}{cc} \frac {\sqrt{2}}2 & 0 \\ 0 & \sqrt{1-\ell^2} \end{array}\right),  \quad \ell \in \R,
$$
and
$$
V_1\in \left\{  \left( \begin{array}{cc} I & O\\ O & \pm I  \end{array}\right),
\left( \begin{array}{cc} I & O\\ O & \pm S  \end{array}\right) \right\},\quad
V_2\in \left\{
\left( \begin{array}{cc} I & O\\ O &  I  \end{array}\right),
\left( \begin{array}{cc} I & O\\ O & S \end{array}\right) \right\}.
$$
The Chui-Lian multi-wavelets \cite{ChuiLian} correspond to the choice $\ell=-\frac {\sqrt{14}}4$ and $b=\frac 12$ in \eqref{eq:Bmatrix}.


The next example,  is related to a special type of multi-wavelet systems proposed in
\cite{LebrunVetterli}. By similar argument as the ones used in \cite{BCS,CLS}, the authors in \cite{LebrunVetterli} derive
proper pre-filters associated to any multi-wavelet basis. The construction is based on the requirement that the mask
$\pb$ preserves the constant data which make any pre-filtering step obsolete. Preservation of constants is equivalent
to the choice $v=(1,1)^T$ in \eqref{eq:y}-\eqref{eq:y1}. In order to reduce the degrees of freedom, we impose some
symmetry constraints on $\pb$ and $\qb$ (see \cite{LebrunVetterli}) directly on the matrix $U$.
Thus, we split $U$ into four symmetric blocks. One of such matrices $U$ (the other possibilities differ only by sign changes)
is given by
$$
 U=\left( \begin{array}{cccc}
\ell& \frac 12 -\ell & \frac {\sqrt2}4 (1+J_\ell^-)& \frac {\sqrt2}4 (1-J_\ell^-) \\
\frac 12 -\ell & \ell & \frac {\sqrt2}4 (1-J_\ell^-)& \frac {\sqrt2}4 (1+J_\ell^-) \\
-\frac 12 &-\frac 12 & \frac 12 & \frac 12 \\
\frac {2\sqrt{2}\ell +1}{2(8\ell^2-1)}J_\ell^+J_\ell^-
& -\frac {2\sqrt{2}\ell +1}{2(8\ell^2-1)}J_\ell^+J_\ell^-&
-\frac 12 J_\ell^+&\frac 12 J_\ell^+
\end{array}
\right)
$$
with $J^\pm_\ell=\sqrt{1\pm 8\ell^2\mp 4\sqrt{2}\ell}$.
The masks $\pb$ and $\qb$ in \cite{LebrunVetterli} correspond to $\ell=\frac {\sqrt {2} }8 \left( 2-\sqrt {7} \right) $
and to the value $b=1 $ in \eqref{eq:Bmatrix}.

\subsection{Wavelets with support $[0,5]$.} \label{sec:examples_N_2}

In this section, we consider the case $K=M=1$ and $N=2$ and apply  the method for determining
the masks of Daubechies (D6) given by Theorem \ref{th:main}.
By such theorem,  the unitary matrix
$$
 U^*=\left(\begin{array}{cc} F_0+F_2 & F_1 \\ F_1 & F_0+F_2\end{array} \right)
$$
contains already all the information about the unitary $F(z)$ we aim to determine. Imposing
the sum rules/vanishing moments of order $3$ on $U^*$ leads to
\begin{eqnarray*}
  F_0&=& \sqrt2\left(\begin{array}{cc}
  	p_0&q_0\\
  	p_1&q_1 \end{array} \right), \\ \\
	F_1&=& \sqrt{2} \left(\begin{array}{cc}
		\frac 18 -4p_0+2p_1 & \frac 18-4q_0-2q_1 \\
  \frac 38 -2p_0& -\frac 38+2q_0  \end{array} \right), \\  \\
	F_2&=& \sqrt{2}\left(\begin{array}{cc}
		\frac 38 +3p_0-2p_1 & \frac 38 +3q_0+2q_1\\
		\frac 18 +2p_0-p_1 & -\frac 18 -2q_0-q_1
		 \end{array} \right), \quad p_0,p_1,q_0,q_1 \in \R.
\end{eqnarray*}
The condition $U^*U=I$ reduces the $4$ parameters to one, $t=q_0$. This requires us to solve
$4$ quadratic equations in $4$ unknowns. We obtain four possible solutions depending on $t$.
We present only one of them that corresponds to the Daubechies wavelets $D6$. The others are the same up to a sign change.
$$
 p_0=\frac 1{16}(1 + a_t), \,
p_1=\frac 18 (2-8t + a_t), \,
q_1=-\frac 1{16} (1+8t-a_t),
$$
where $a_t=\sqrt{-255t^2+32t+7}$. The parameter $t$ is determined by solving one equation
with the radical $a_t$ and yields $t= \frac 1{32}+\frac1{32}\sqrt{10}\mp\frac 1{32}\sqrt{5+2\sqrt{10}}$.

\section{Potapov-Blaschke factorizations: scalar case} \label{examples:potapov}

In this subsection, we consider only the cases $K=M=1$ and $N=1,2$. We think the corresponding examples
are sufficient for the comparison of our results with the ones in \cite{AlpayJorgensen}. The case of $N=1$ is of special interest as it directly establishes a link between our results and the results in \cite{AlpayJorgensen}.

It has been observed already in \cite{Potapov} , see also e.g. \cite[Theorem 4.3]{Gohberg}, that any trigonometric polynomial of degree $N$,
which is unitary on the unit circle, possesses a factorization into so-called Blaschke-Potapov factors. These factorizations were
applied for constructions of finite impulse response filters in \cite{AlpayJorgensen}. In the case $N=1$, the result of
Lemma \ref{lem:N_1} also leads us to Blaschke-Potapov factors. Indeed in this case $A=DU^*=(I-B)U^*$ and $C=BU^*$ and hence,
$$
 F(\xi)=A+Cz=DU^* + BU^* \xi = (I-B)U^* + BU^* \xi = (I-B+B\xi)U^*, \quad |\xi|=1.
$$
For factorizations of higher degree $F$ into Blaschke-Potapov factors we use the matrices $B$ and $U$ constructed via the algorithm in
$(II)$ Lemma \ref{lem:N_1}. In general, any unitary $F(\xi) \in \C^{2M \times 2M}$, $|\xi|=1$, of degree $N$ possesses a factorization
$$
 F(\xi)=\prod_{j=1}^N =(I-B_j+B_j \xi)U_j, \quad U_j^*U_j=I, \quad j=1, \ldots, N,
$$
where $B_j$ are some rank-1 projections.

%
%

For $N=2$, the Daubechies (D6) scaling and wavelet masks are obtained by considering
$$
 F(\xi)=(I-B_1+B_1 \xi) (I-B_2 +B_2 \xi)\,U^*
$$
for some $B_1$ and $B_2$ in \eqref{eq:Bscalar}. To determine the corresponding parameters $b_1$ and $b_2$,
as mentioned above, we determine the corresponding $F_j$ and, then, impose further the sum rules/vanishing
moments of order $2$
$$\sum_{j=0}^{2}
 \left(\begin{array}{cc}
 (2j)^k & -(2j+1)^k \\
 (2j)^k & (2j+1)^k
 \end{array}\right)F_j=\left(\begin{array}{cc} 0 & * \\ * & 0
 \end{array}\right), \quad k=1,2,
$$
where $*$ symbolizes the matrix entries that do not contribute to our computations.
We get
$$
b_1=\frac 54-\frac 18\,\sqrt {10}, \quad b_2=\frac 18\,\sqrt {10},
$$
or more explicitly
\begin{eqnarray*}
B_1&=&\left( \begin {array}{cc} 5/4-1/8\,\sqrt {10}&-1/8\,\sqrt {-30+12\,
	\sqrt {10}}\\ \noalign{\medskip}-1/8\,\sqrt {-30+12\,\sqrt {10}}&-1/4+
1/8\,\sqrt {10}\end {array} \right),  \\
B_2&=&\left( \begin {array}{cc} 1/8\,\sqrt {10}&-1/8\,\sqrt {-10+8\,\sqrt {
		10}}\\ \noalign{\medskip}-1/8\,\sqrt {-10+8\,\sqrt {10}}&1-1/8\,\sqrt
{10}\end {array} \right).
\end{eqnarray*}

To obtain $b_1$ and $b_2$, we, additionally, need to solve one quadratic and one equation with the radicals.
Thus, the computational effort is exactly the same as in Section \ref{sec:examples_N_2}.

\section*{Conclusion}\label{sec:conclusion}
In this paper we have shown that results from system theory provide complete characterization of all orthogonal  (multi)wavelet filters. This has been achieved by explicitly determining the structure of the $ABCD$ matrix from which an algorithm for multi-wavelet construction can be derived.
The aim of the paper was not to propose new classes of multi-wavelets, but rather to provide a unifying framework for all the different constructions proposed in literature. 

Future work includes the non-straightforward generalization  to the bivariate case and the use of our results for the explicit construction of new classes of  matrix wavelet filters satisfying more general properties, for example the exponential rather than polynomial vanishing moment property, or customized according to the problem.

\section*{Acknowledgements}
Maria Charina was sponsored by the Austrian Science Foundation (FWF) grant P28287-N35.  Part of the research was carried out  during a visit of the first author at the University of Reggio Calabria supported by GNCS-INdAM.


\end{document}